\documentclass[psamsfonts]{amsart}

\usepackage{amssymb,amsfonts}
\usepackage{soul}
\usepackage[all,arc,color,matrix,arrow]{xy}
\usepackage{enumerate}
\usepackage{mathrsfs}
\usepackage{color}

\newtheorem{thm}{Theorem}[section]
\newtheorem{cor}[thm]{Corollary}
\newtheorem{prop}[thm]{Proposition}
\newtheorem{lem}[thm]{Lemma}

\theoremstyle{definition}

\theoremstyle{remark}

\newcommand{\ZZ}{\mathbb{Z}} 
\newcommand{\QQ}{\mathbb{Q}} 
\newcommand{\CCa}{\mathcal{C}}
\newcommand{\RR}{\mathbb{R}} 
\newcommand{\CC}{\mathbb{C}} 
\newcommand{\PP}{\mathbb{P}} 
\newcommand{\LL}{\mathscr{L}}

\newcommand{\uconf}{\mathrm{UConf}} 
 
\newcommand{\UU}{\mathbb{U}}

\newcommand{\YY}{\mathscr{Y}}
\newcommand{\A}{\mathscr{A}}
\newcommand{\HH}{\mathscr{H}}
\newcommand{\G}{\mathfrak{G}}
\newcommand{\Ua}{\mathscr{U}}
\newcommand{\EE}{\mathbb{E}}
\newcommand{\FF}{\mathbb{F}}
\newcommand{\supp}{\textrm{supp}}
\title{Stable Cohomology of Discriminant Complements for an algebraic curve}

\author{Ishan Banerjee}

\begin{document}

\maketitle
\begin{abstract}
     Let $\LL$ be a degree $n \ge 1$ line bundle on a smooth projective complex algebraic curve $X$. Let $U(\LL)$ (resp. $\UU(\LL)$) denote the set of algebraic (resp. $C^{\infty}$) sections of $ \LL$. We show that:
     \begin{enumerate}
         \item The inclusion $U(\LL) \to \UU(\LL)$ induces an isomorphism between $H^*(\UU(\LL); \ZZ)$ and $H^*(U(\LL); \ZZ))$for  $*<< n$.
         \item $\UU(\LL)$ is aspherical.
         \item $\pi_1(\UU(\LL))$ is a group closely related to the surface braid group $Br_n(X)$.
     \end{enumerate}
\end{abstract}

\section{Introduction}
In this paper we are concerned with understanding the difference between spaces of algebraic and $C^{\infty}$ sections of complex line bundles on a smooth projective algebraic curve $X$ over $\CC$. We prove that these spaces have isomorphic cohomology in a range of degrees that grows with the degree of the line bundle.

 Let $X$ be a smooth projective algebraic curve over $\CC$ of genus $g$. Let $\LL$ be an algebraic line bundle on $X$ of degree $n$.
 
 Let $C^{\infty}(X,\LL)$ be the vector space of \emph{smooth} sections of $\LL$. Given $s\in C^{\infty}(X,\LL)$, we say $p \in X$ is a \emph{regular zero} of $s$ if $s(p) = 0$ and $s'(p) \neq 0$. If $s \in C^{\infty}(X, \LL)$ is an algebraic section,  then being a regular zero is equivalent to having index 1. Let $$\UU(\LL) = \{ s \in C^{\infty}(X,\LL) |\textrm{all zeroes of } s \textrm{ are isolated and of index }1\}.$$
 
 Let $$U(\LL) := \{ s \in H^0(X,\LL) : \textrm{All zeroes of } s\textrm{ are regular}\}.$$

There is a natural inclusion map $i: U(\LL) \hookrightarrow \UU(\LL)$.  The aim of the present paper is to understand this inclusion at the level of cohomology. Our main theorem is as follows:
\begin{thm}\label{main}
Let $X$ be a smooth projective complex algebraic curve of genus $g$. Let $n \ge 0$. Let $\LL$ be an algebraic line bundle of degree $n$ on $X$. Let $i: U(\LL) \to \UU(\LL)$ be the inclusion map. Then for all $0\le 2k \le n -2g $,
$$i^*:H^k(\UU(\LL); \ZZ) \to H^k(U(\LL);\ZZ)$$ is an isomorphism.

\end{thm}



We also provide some qualitative understanding of the topology of the space $\UU(\LL)$ and relate it with more classical objects. 
Let $n = \textrm{deg}(\LL)$. Given a space $M$, define $$\textrm{PConf}_n M := \{(x_1, \dots ,x_n) \in M^n | x_i \neq x_j\}.$$ The permutation action of $S_n$  on $M^n$ restricts to an action on $\textrm{PConf}_n M$.
Let $$\uconf_n M = \textrm{Pconf}_n M /S_n$$ be the \emph{unordered configuration space} of $n$ points on $M$.

Define  the $n$ stranded \emph{surface braid group} on a surface $X$ as $$Br_n(X) := \pi_1(\uconf_n(X)).$$ For our purposes we will need to define a group  $\tilde Br_n(X)$ that we call the extended surface braid group. This will be defined later on in Section 3 as $\pi_1(U_n^{alg})$ where $U_n^{alg}$ is a space defined in Section 3 that is a $\CC^*$ bundle over $\uconf_n X$.   1

 Let $\pi: \tilde Br_n(X) \to Br_n(X)$ be the projection.
Now, $\uconf_n(X)  \subseteq Sym^n (X)$ has an Abel Jacobi map  to $Pic^n(X)$. This induces a map $\alpha: Br_n(X) \to \ZZ^{2g}$.
Let $\A = \alpha \circ \pi $.
Let $K_n \subseteq  \tilde Br_n(X)$ be the kernel of this map.
\begin{thm}\label{main2}
Let $n \ge 1$. Let $X$ be a smooth projective curve. Let $\LL$ be a line bundle of degree $n$ on $X$.
The space $\UU(\LL)$ is a $K(\pi,1)$. Furthermore, $$\pi_1(\UU(\LL)) \cong K_n.$$ 

\end{thm}

\vspace{2pt}
\textbf{Motivation}
In \cite{VW} Vakil and Wood consider (among other things) the 'stable class' of the discriminant locus in the Grothendieck group of varieties $K_0(\textrm{Var})$. 
Let us recall the definition of the \emph{Grothendieck group of varieties}. Let us fix a base field $k$. Then we can consider the set $$\textrm{Var}_k = \{X : X \textrm{is a variety over } k\} / isomorphism.$$ We can form a monoid $M$ out of $Var_k$ as follows: let $M$ be generated by elements of $Var_k$, with the relation, if $Y \subseteq X$, $[X] = [X -Y] +[Y] \in M$. The \emph{Grothendieck group} $K_0(\textrm{Var}_k)$ is the group completion of $M$. It has a ring structure coming from the product of varieties. In the literature, the element $\mathbb{A}^1$ (often denoted $\mathbb{L}$) is sometimes inverted. Define $\mathscr{M}_{\mathbb{L}} = K_0(\textrm{Var}_k) [\frac{1}{\mathbb{L}}]$.

Consider a smooth variety $X$ along with an ample line bundle $\LL$ on it.
\begin{thm}[Vakil - Wood \cite{VW}] \label{VW}
Let $j \ge 1$. Let $U(\LL^{\otimes j}) $ be the (open) variety of sections with smooth zero locus. Let $\zeta_X$ be the Kapranov motivic zeta function, and let $d$ be the dimension of $X$. Then,
$$
\lim_{j \to \infty} \dfrac{[U(\LL^j)]}{[H^0(X,\LL^j)]} = \frac{1}{\zeta_X(d+1)}.
$$
Here the limit is with respect to the dimension filtration in $\mathscr{M}_{\mathbb{L}}$.
\end{thm}
While Theorem \ref{VW} seems to have nothing to do with the cohomology of the space $U(\LL^j)$, there is a specialisation map $$K_0(\textrm{Var}_k) \to \{\textrm{Weighted Euler characteristics}\}.$$ Thus, Theorem 1.3 implies that there is a stabilisation of Euler characteristics and one can hope for a stabilisation in cohomology as well.

 In \cite{T}, Tommasi proves a cohomological result in the same vein as that of this paper, where she studies discriminant complements on $\PP^n$.
 Her set up is as follows. Let $d,n \ge 1$. Let $X = \PP^n$. Let  $\LL = \mathcal{O}(d)$. Let $$U(\LL) = \{f \in H^0(X, \LL) | f \textrm{ has only regular zeroes}\}.$$ Then the main theorem of \cite{T} stated in our notation is as follows:

 \begin{thm}[Tommasi \cite{T}]
  Let $d,n \ge 1$.
  Let $X = \PP^n$, $\LL = \mathcal{O}(d)$.
  Let $0\le k \le \frac{d+1}{2}$.
  Then $$H^k(U(\LL); \QQ) \cong H^k(GL_{n+1}(\CC); \QQ).$$
 \end{thm}
 
Our motivation for the present paper was to understand if there are  stability phenomena for discriminant complements over general varieties and whether cohomology in the stable range is dependent only on the topology of the variety. Theorem \ref{main} shows that at least in the case of an algebraic curve there is some  kind of stability phenomenon with cohomology in the stable range being purely topological in nature. We are currently working on extending these results to more general varieties.

\vspace{2pt}
\textbf{Relation to other work}: Orsola Tommasi has anounced some results on homological stability for discriminant complements over arbitrary smooth projective varieties. We believe that the results in this paper are  substantially different from hers. We focus on relating discriminant complements to spaces of $C^{\infty}$ sections, which is not the focus of her results.

\vspace{2pt}
\textbf{Acknowledgements}: I would like to thank my advisor Benson Farb for his patience and help with editing countless versions of this paper. I'd like to than Ravi Vakil and Madhav Nori for answering several of my questions, without which I would not have been able to prove these theorems.

\section{Smooth sections}
The space $\UU(\LL)$ is actually easy to understand topologically. 
There is a fibration $\pi: \UU(\LL) \to \uconf_n X$ defined by: $\pi(f) =\{a | f(a) =0\}$. 
We shall need to understand the fibres $\pi^{-1}(\{a_1, \dots a_n\}) \subseteq \UU(\LL)$, but first we shall introduce some basic objects and prove some more technical lemmas.
Let $Y,Z$ be based spaces. Let $C(Y,Z), C_*(Y,Z)$ denote the space of continuous maps from $Y$ to $Z$ and the space of based continuous maps from $Y$ to $Z$.
Let $X$ be a $C^{\infty}$ manifold. Let $\G = C^{\infty}(X, \CC^*)$. For $p \in X$,
let $$\G_p = \{f \in \G | f(p) =1\}.$$
Before we begin with stating and proving the propositions in this section, we note that they are mostly applications of the fact that $\CC^*$ is a $K(\pi,1)$ space and is covered by $\CC$ a contractible space. The space of continuous based maps into a $K(\pi, 1)$ has been classically studied first by Thom and then by many others. Lemma \ref{contG} is a bit more specific to our situation and is not an immediate application of the theory of $K(\pi,1)$ spaces.
\begin{prop}\label{Ghtpy}
with the above notation,
\begin{enumerate}
    \item $\G_p$ is weak  homotopy equivalent to $H^1(X, \ZZ)$.
    \item $\G$ is weak homotopy equivalent to $H^1(X, \ZZ) \times \CC^*$.
\end{enumerate}
\end{prop}
\begin{proof}
The space $\CC^*$ is a $K(\ZZ,1)$ and by the Proposition labelled Thom [4] in \cite{Hae} $C_*(X, \CC^*)$ is homotopy equivalent to $H^1(X, \ZZ)$ (i.e. each of its components is contractible and the set of components is in natural bijection with $H^1(X,\ZZ)$). By Theorem 1.5 of \cite{KM}, $C_*(X, \CC^*)$ is weak homotopy equivalent to $C^{\infty}(X, \CC^*)$. This establishes (1).
To establish (2) we note that $\G$ is homeomorphic to $\G_p \times \CC^*$, indeed an explicit homeomorphism is given $(f, \alpha) \in \G_p \times \CC^* \mapsto \alpha f \in \G$.

\end{proof}

\begin{prop}\label{contdisk}
Let $D = \{z \in \CC | |z| \le 1\}$. Let $$S = \{f \in C^{\infty} (D - \{0\}, \CC^*)| f(z) = 1 \textrm{ for } z\in S^1 \}.$$ Then $S$ is contractible. 
\end{prop}
\begin{proof}
Let $$S' = \{f \in C^{\infty}(D -\{0\}, \CC)| f(z) = 0 \textrm{ for } z \in S^1\}.$$
Then $S'$ deformation retracts to the constant function $f(z) =0$ by a straight line homotopy.

Now note that there is a covering map $\exp: \CC \to \CC^*$ such that $exp(0) = 1$. We note that if $f \in S$, $f$ is nullhomotopic. Hence by the lifting criterion, there is a unique lift $\tilde f : D - \{0\} \to \CC$ such that $\tilde f(1) = 0$ and $\exp \circ \tilde f = f$.
We know have a homeomorphism between $S$ and $S'$ $\phi: S\to S'$be defined by $\phi (f) = \tilde f$.

This implies that $S$ is contractible.
\end{proof}

\begin{lem}\label{sline}
Let $D$ be the closed unit disk in $\CC$. Let $$F = \{f \in C^{\infty}(D - \{0\}, \CC^*) | \: f |_{\partial D} =1\}.$$
Let $$\tilde F = \{f \in C^{\infty}(D - \{0\}, \CC^*) | \: f \textrm{ is nullhomotopic and } f(1)=1 \}.$$ Then $\tilde F$ deformation retracts to the point $f_0$, where $f_0(x) =1$ for all $x$. Furthermore, the deformation retraction preserves the subset $F$.
\end{lem}
\begin{proof}
Given $f \in \tilde F$ we can lift it to a  unique map $\tilde f :D- \{0\} \to \CC$, such that $\exp \circ \tilde f = f $ and $\tilde(f)(1) = 0$. The straight line homotopy in $\CC$ defines a homotopy between $\tilde f$ and the constant function. This in turn defines a homotopy $h$ between $f$ and $f_0$. This gives us our deformation retraction. It is easy to see that this preserves $F$.
\end{proof}

\begin{lem}\label{contG}
Recall that $\G= C^{\infty}(X, \CC^*)$. Let $n \ge 1$. Let $\{a_1, \dots, a_n \} \in \uconf_n X$.
Then, 
\begin{enumerate}
    \item There is a free action of $\G$ (as a group under multiplication) on $\pi^{-1}(\{a_1, \dots, a_n \})$.
    \item The quotient $\pi^{-1}(\{a_1, \dots, a_n\})/ \G$ is contractible.
\end{enumerate}
\end{lem}
\begin{proof}
We define our action as follows: if $s \in \pi^{-1}(\{a_1, \dots, a_n \}) $ and $g \in \G$,  define $g.s(x) = g(x) s(x)$. If $g.s(x) =0$ then $s(x) =0$ as $g(x) \neq 0$ for all $x \in X$. Furthermore, if $g.s = s$ then $g(x) = 1$ for $x \in X - \{a_1, \dots, a_n\}$ and since  $X - \{a_1, \dots a_n\}$ is dense, $g(x) =1$ for all $x \in X$. This concludes the proof of (1).

Let $D_{i}$ be small disks surrounding the points $a_i$. Let $G_i= \{f \in C^{\infty}(D_{i}, \CC^*)| f|_{\partial D_i} =1\}$.
We can identify each $G_i$ with the set of based maps from $S^2$ to $\CC^*$. Since $\pi_1(S^2) = 0$, any based map $f: S^2 \to \CC^*$ lifts to a unique map $\tilde f: S^2 \to \CC$. Hence $G_i$ is homeomorphic to the space of based maps from $S^2$ to $\CC$ and since $\CC$ is contractible, $G_i$ is contractible.

Let $F_i = \{f \in C^{\infty}(D_{i} -a_i, \CC^*) :\:  f|_{\partial D_i} =1\}$. By Proposition \ref{contdisk} $F_i$ is contractible. Let $$\tilde F_i = \{f \in C(D_{i} -a_i, \CC^*) | f|_{\partial D_i} \textrm{ is nullhomotopic}\}.$$ Let $$\tilde G_i = \{f \in C(D_i, \CC^*)\}.$$ Since $D_i$ is contractible, the space $\tilde G_i \simeq \CC^*$ and the quotient $\tilde F_i/\tilde G_i$ is contractible (this is analogous to the proof of Proposition \ref{contdisk}). There is an inclusion map $i: F_i/G_i \hookrightarrow \tilde{F}_i / \tilde{G}_i$ which is a homotopy equivalence as both spaces are contractible. By Lemma \ref{sline},there is a  map $j: \tilde F_i/ \tilde G_i \to F_i /G_i$ satisfying the following properties.

\begin{enumerate}
   
    \item There exists a homotopy $h: \tilde F_i\times [0,1] \to \tilde F_i$ such that $h(f,1) = f$, $h(f,0) =j(f)$ and $h(f,t)|_{U_i} =f|_{U_i}$.
    \item For all $f \in F_i$, $h(f,t) \in  F_i$.
    
\end{enumerate}

Let $ A =\pi^{-1}(\{a_1, \dots a_n\})/ \G$. Fix an $s_0 \in pi^{-1}(\{a_1, \dots a_n\})$. Let $\phi: A \to \prod_{i =1}^n \tilde F_i/ \tilde G_i$ be defined as follows: $$\phi (s) = (s/s_0|_{D_1 -a_1}, \dots s/s_0 |_{D_{n} -a_n}).$$ We claim that $\phi$ is a homotopy equivalence. To prove this we first define  $\psi: \prod_{i =1}^n F_i/G_i \to A$ as follows:.

\begin{equation*}
    \psi(f_1, \dots f_n)(x) = 
    \begin{cases}
    f_i(x)s_0(x) & \textrm{if } x\in D_i\\
    s_0(x) & \textrm{otherwise.}
    \end{cases}
\end{equation*}

It is easy to see that $\psi \circ j\circ \phi \simeq \textrm{Id}$ (the homotopy $h$ mentioned above can be seen to define such a homotopy).  Since $\prod_{i =1}^n F_i/G_i$ is contractible this implies (2). 
\end{proof}

\textbf{Remark:}
The action of $\G$ on $\pi^{-1}(\{a_1, \dots, a_n\})$ is in fact  \emph{not} transitive for any value of $n \ge 1$. The following example will illustrate this fact. Let $D$ be the closed unit disk in $\CC$ which we identify with $\RR^2$. Let $f: D\to \CC$ be defined as $f(x,y) =(x,y)$. Let $g: $ $g(x,y) = (2x,y)$. Let $E \to \PP^1$ be the unique degree 1 line bundle on $\PP^1$. Let $\phi: E|_{D} \to D \times\CC$ be a trivialisation. Let $\bar f, \bar g$ be $C^{\infty}$ sections on $\PP^1$ of $E$ such that for $(x,y) \in D$, $\phi(\bar f (x,y)) = ((x,y), f(x,y)) $ and $\phi(\bar g (x,y)) = ((x,y), g(x,y)) $ (it follows from a standard obstruction theoretic argument that there indeed exist such  $\bar f$ and $\bar g$). If there exists $\bar h \in C^{\infty}(X, \CC^*)$ such that $\bar h \bar f = \bar g$, then $\bar h(0,0) = \lim_{(x,y) \to (0,0)}\frac{g(x,y)}{f(x,y)}$. But this limit does not exist and hence $\bar f$ and $\bar g$ are not in the same $\G$ orbit.

\begin{cor}\label{isKG1}
Let $n \ge 0$. Let $X$ be a smooth projective curve and $\LL$  a line bundle on it of degree $n$. Then $\UU(\LL)$ is a $K(\pi,1)$.
\end{cor}

\begin{proof}
There is a fibration

\centerline{\xymatrix{
\pi^{-1}(\UU(\LL)) \ar[r] & \UU(\LL) \ar[d]^{\pi}\\
& \uconf_n X\\
}}

 Lemma \ref{contG} implies that$\pi^{-1}(\UU(\LL)) \simeq \G$. The space $\G$ is a $K(\pi,1)$ by Proposition \ref{Ghtpy}. Since $\UU(\LL)$ is the total space in a fibration with both base and fibre $K(\pi,1)$ spaces is itself a $K(\pi,1)$.

\end{proof}

\section{Abel-Jacobi} 
 
 In this section we will try to understand the space $U(\LL)$.
 Our method to understand the topology of $U(\LL)$  is by making it a subspace of  a  space $U_n^{alg}$ which we shall construct.

  We would like to remind the reader that to give a complex line bundle $\LL$, a holomorphic structure $h$ is equivalent to giving a Dolbeault operator $\partial_h: \Gamma(L) \to \Omega^{0,1} \otimes \Gamma(L)$. More details on Dolbeault operators and holomorphic structures may be found in Ch.3 of \cite{AM}. Let $ \HH_n $ be the space of holomorphic structures on $\LL$. The group $\G_p$ acts on $\HH_n$ with trivial stabilizers. The quotient $\HH_n / \G_p$ is naturally isomorphic to $\textrm{Pic}_n X$.
 
 Let $$\Ua_n = \{(s,h)\in \UU(\LL) \times \HH_n  | s \textrm{ is a algebraic section } of \LL \textrm{ with respect to } h \} .$$ Note that the groups $\G$ and $\G_p$ act on this space $\Ua_n$.

 Let $ U_{n}^{alg} := \Ua_n/ \G_p$.
 There is a surjective map $\pi :\Ua_n \to \HH_n$ defined by $\pi(s,h) = h$. Since $\pi$ is equivariant with respect to the action of $\G_p$, it descends to a surjection  $$\A: U_{n}^{alg}= \Ua_n/ \G_p \to \HH_n /\G_p = \textrm{Pic}_n X.$$ We observe that for $\LL \in \textrm{Pic}_n X$, we have the equality   $ \A^{-1}(\LL) = U(\LL)$.  This map $\A$ can be seen as a section level version of the Abel-Jacobi map. We now wish to understand the topology of $U_n^{alg}$.
 \begin{prop}
 Let $n \ge 1$.
 \begin{enumerate}
     \item  $U_n^{alg}$ is a $K(\pi,1).$ 
     \item There is a short exact sequence
     $$1 \to \ZZ \to \pi_1(U_n^{alg}) \to Br_n(X) \to 1.$$
 \end{enumerate}
 \end{prop}
\begin{proof}
There is a fibration $\pi: U_n^{alg}: \to \uconf_n X$ defined by $$\pi(s,h) = \{a \in X | s(a) =0 \}.$$
If $\bold{a} = \{a_1 \dots a_n\} \in \uconf_n X$, then $\pi^{-1}(\bold{a}) \cong \CC^*$, as algebraic sections of a line bundle are uniquely identified with their zeroes up to a scalar. Since $\CC^*$ and $\uconf_n X$ are $K(\pi,1)$ spaces, so is $U_n^{alg}$.
\end{proof}

\subsection{An alternative definition of $U_n^{alg}$}
In this subsection we will give an alternative definition of $U_n^{alg}$. This will not be used in the rest of the paper.

Let $X$ be a smooth projective curve of genus $g$. Let $n>g$. Let $Sym^n X$ be the $n$th symmetric power of $X$. Let $\mathcal{P}$ denote the Poincare line bundle on $X \times \textrm{Pic}_n X$, this is the unique line bundle on $X \times \textrm{Pic}_n (X)$ such that $\mathcal{P} |_{X \times \{\LL\}} = \LL$ and $\mathcal{P}|_{\{p\} \times \textrm{Pic}_n X}$. Let $\pi:X \times \textrm{Pic}_n X \to \textrm{Pic}_n X$ denote the projection. The pushforward $\pi_*(\mathcal{P})$ defines a vector bundle $E$ on $\textrm{Pic}_n X$, sometimes called the Picard bundle. Let $E_0 \subseteq E$ denote the zero section. We may identify $Sym^n X$ with $E - E_0 / \CC^*$, i.e. $Sym^n X$ is the  projective space bundle associated to the vector bundle $E$. Let $\rho: E- E_0 \to Sym^n X$ denote the projection map. We then define $U_n^{alg}$ to be $\rho^{-1} (\uconf_n X)$. 

Let us emphasize that $E- E_0 \to Sym^n X$ is not a trivial $\CC^*$ bundle. Indeed after restricting to a fibre of the projection $Sym^n X \to \textrm{Pic}_n X $ the bundle $\rho$ restricts to the bundle $\CC^{n-g +1} - \{0\} \to \PP^{n-g}$ which is classically known to be non trivial. While it is possible that the bundle $U_n^{alg} \to \uconf_n X$ is a trivial $\CC^*$ bundle, we are unable to determine whether this is the case.

\section{Comparing different fibres}
To understand $U(\LL)$ we will analyze the map $\A: U_n^{alg} \to \textrm{Pic}_n X$.
We will prove that the map $\A$ is similar to a homology fibration. More precisely, we have the following.
\begin{thm}\label{AisHF}
Let $n \ge 0$. Let $X$ be a smooth projective  complex algebraic curve of genus $g$, $\LL$ a line bundle of degree $n$ on $X$. Let $2k \le n-g$. Let $\A$ be the map defined in Setion 3.
Let $W \subseteq Pic_n (X)$ be a small contractible  neighbourhood of $\LL$ homeomorphic to a ball.
Let  $$i: \A^{-1}(\LL) \hookrightarrow \A^{-1}(W)$$ be the inclusion map. 
Then $$i^*: H^k(\A^{-1}(W); \ZZ) \to  H^k(\A^{-1}(\LL); \ZZ )$$ is an isomorphism.
\end{thm}
Before embarking on the proof of Theorem \ref{AisHF} we will need to set up some machinery. 

There is a vector bundle $\pi: H^0(X,W) \to W$ defined as follows. Let $$H^0(X,W) = \{(s, \LL) | \LL \in W , s\in H^0(X, \LL)\}.$$ Then $\A^{-1}(W)$ is an open subset of $H^0(X,W)$. The topology of the complement $\Sigma_W = H^0(X,W) - \A^{-1}(W)$ will be important for us to understand. It is immediate that $\Sigma_W = \{(f, \LL) | \LL \in W, f \in \Sigma(\LL)\},$ since for any $\LL \in W$ $\A^{-1}(\LL)$ consists of all sections of $\LL$ with regular zeroes.

We will create a relative stratification of $\Sigma^{-1}(W)$. This is similar to the stratification in \cite{T}. 
Let $$\Sigma_{\LL}^{\ge k} = \{f \in \Sigma_{\LL} |\textrm{ } |\textrm{Sing}(f)| \ge k\}.$$

Let $N = \frac{d -g}{2} $. We stratify  $\Sigma_{\LL}$, the complement of $\A^{-1}(\LL)$ in $ H^0(X, \LL)$ by $$\Sigma_{\LL}^{\ge k} = \PP \{f \in  V|  f \textrm{ has  atleast } k \textrm{ distinct singular zeroes}\},$$
 for $k \le N$. So $\Sigma_{\LL}^{\ge 1} \supset \Sigma_{\LL}^{\ge 2} \supset \dots$.

Now we construct a cubical space  $C$ that will be involved in understanding $\Sigma(\LL)$. Let $N = \frac{d-1}{2}$. Let $I$ be a subset of $\{1, \dots, N-1\}.$ Let $I = \{i_1, \dots, i_k \}$ let  $$C_I :=\{(f, x_{1}, \dots, x_k)| f \in \Sigma(\LL), x_j \in \uconf_{i_j}(X) \textrm{ } x_1 \subseteq x_2, \dots, \subseteq x_k \subseteq \textrm{ Singular zeroes of  }f \}.$$
We define $$C_{I \cup \{N\}}:= \{(f, x_1, \dots, x_k) \in C_I|  f\in \bar \Sigma^{\ge N}\}. $$ If $I \subseteq J$ then we have a natural map from $C_J \to C_I$ defined by restricting $p$. This gives $C$ the structure of a cubical space over the set $\{1, \dots, N\}$. We can take the geometric realization of $C$ denoted by $|C|$. Then there is  a map $\rho: |C| \to \Sigma(\LL)$, induced by the forgetful maps $C_I \to \Sigma(\LL)$.

$|C|$ is topologized in a non-standard way. The topology we give is analogous to the topology on $\mathscr{X}$ in \cite{T}. The primary reason we give $|C|$ this topology is to make $\rho$ proper.
For $k< N$, there is an inclusion $$i:\uconf_k(X) \to Gr(h^0(X,\LL) - 2k,H^0(X,\LL)).$$ We define $L_k( \LL)$ to be the Zariski closure of the image. We will omit the $\LL$ in our notation if there is only one line bundle that we are discussing. There is a relation, $<$ on the collection of all $L_k$, defined by $\lambda_1 < \lambda_2$ if as subspaces of $H^0(X,\LL)$, $\lambda_2 \subseteq \lambda_1$. Note that this extends the relation $\supset$ on the collection of all $\uconf_k(X)$.
Let $I =\{i_1, \dots, i_k\} \subseteq \{1, \dots, N-1\}$. Let $\bar C_I = \{(f, \lambda_1, \dots, \lambda_k) |\lambda_j \in L_{i_j}, \lambda_1 < \lambda_2 \dots < \lambda_k < \textrm{Sing}(f)\}$. Let $\bar C_{I \cup N} = \{f, \lambda_i , \dots \lambda_k \in C_I| f \in \Sigma^{\ge N}_{\LL}\}.$ Then  $\bar C$ forms a cubical space in the same way that $C$ does.

Take the geometric resolution $|\bar C|$. Now we will construct a map $|\bar C| \to |C|$ that is the identity on $|C| \subseteq |\bar C|$. This will exhibit $|C|$ as a quotient of $|\bar C|$ and we will give it the quotient topology.
Given $\lambda \in L_k$, we can define $\textrm{supp}(\lambda) \in \uconf_{n(\lambda)}(X)$ by $\supp(\lambda) = \cap_{f \in \lambda}\textrm{Sing}(f).$

This defines a map $\textrm{supp}: |\bar C| \to |C|$ given by $$(f, \lambda_i,  s_i) \in \bar C_{I} \times \Delta_I  \mapsto (f, \supp(\lambda_i), s'_j) \in C_{I} \times \Delta_J.$$ Here $J =\{\supp (\lambda_i)\}$ and $s'_j = \sum _{n(\lambda_i) =j}s_i$.

The maps $\bar C_ I \to \Sigma_{\LL}$ are proper and hence so is the induced map  $|\bar C| \to \Sigma_{\LL}$.

\begin{prop}\label{isequiv}
The map $\rho: |C| \to \Sigma$ is a proper homotopy equivalence.
\end{prop}
\begin{proof}
In our setting having proper contractible fibres implies that the map $\rho$ is a proper homotopy equivalence, this follows by combining Theorem 1.1 and Theorem 1.2 of \cite{L}.  We note that if $f \not \in \bar \Sigma^{\ge N}$ the fibre $\rho^{-1}(f)$ is the simplex with vertices labelled by the singular points. If $f \in \bar \Sigma^{\ge N}$ then the fibre is a cone.  We have given $|C|$ the quotient topology.The map $\rho$ is a factor in the composite $|\bar C| \to |C| \to \Sigma_{\LL}$ which is proper, hence $\rho$ itself is proper.
\end{proof}

\vspace{2pt}
Now as in any geometric realization, $|C|$ is filtered by $$F_n = \mathrm{im} (\coprod_{|I| \le n} C_I\times \Delta_k).$$ The $F_n$ form an increasing filtration of $|C|$ , i.e. $F_1 \subseteq F_2 \dots F_n \subseteq F_{n+1}\subseteq \dots$ and $\cup_{n=1}^{\infty}F_n = |C|$.

We define $$B_n = \{f \in \Sigma_{\LL}| \: f \textrm{ has at least  } n \textrm{ singular zeroes}\}.$$

\begin{prop}\label{bundle}
 Let $n <N$. Let $\Delta_n^ {\circ}$ be the interior of an $n$ simplex. The space $F_n - F_{n-1}$ is a $\Delta_n^\circ$- bundle over the space $B_n$. This is in turn a vector bundle over $\uconf_n(X)$. 
\end{prop}
\begin{proof}
The fact that $B_n$ is the total space of a vector bundle over $\uconf_n(X)$ follows from Riemann-Roch, the fibres are all vector subspaces of $H^0(X, \LL)$ of codimension exactly $2(n+1)$.
A point in $F_n - F_{n-1}$ is a pair $((f, x_0, \dots x_n), s_0 ,\dots s_n) $ where the $s_i$ are the simplicial coordinates. We have a map $ \pi:F_n - F_{n-1} \to B_n$ defined by $$(f, (x_1, \dots, x_n ), (s_0, \dots, s_n) ) \mapsto (f, (x_1, \dots, x_n )).$$ The map $\pi$ expresses $F_n - F_{n-1}$ as a  $\Delta_n^\circ$ bundle over $B_n$.
\end{proof}
\vspace{2pt}

 Let $e_d = \mathrm{dim}_{\CC}(H^0(X,\LL)) $.  
We define a local coefficient system on $\uconf_n X$, denoted  by $\pm \ZZ$, in the following way. There is a homomorphism $\pi_1(\uconf_n X) \to S_n$  associated to the covering $\mathrm{PConf}_n X \to \uconf_k X$. We compose this with the sign homomorphism $S_n \to \pm 1 \cong GL_1(\ZZ)$ to obtain our local system on $\uconf_n X$.
\begin{prop}\label{cohocomp}
Let $ d \ge 1$. Let $n<N$.  $$\bar H_*(F_{n+1} - F_n; \ZZ)= \bar H_{*- (e_d -(2(n+1)))}(\uconf_{n+1}(X); \pm \ZZ).$$
\end{prop}
Proof: By Proposition \ref{bundle} the space $F_k - F_{k-1}$ is a bundle over $\uconf_k(X)$. This fact implies that $$\bar H_*(F_k - F_{k-1}) \cong H_{*-(k + 2e_d -2(n+1)(k+1))}(\uconf_k(X),  \ZZ(\sigma)).$$ Here $\ZZ(\sigma)$ is the local sytem obtained by the action of $\pi_1(\uconf_k(X))$ on the fibres $\bar H_k(\Delta_k^{\circ}, \ZZ)$ where in this case $\Delta_k^{circ}$ is the open $k$ simplex corresponding to the fibres of the map $F_k - F_{k-1} \to B_k$. But one observes that the action of $\pi_1(\uconf_k(X))$ on this open simplex is by permutation of the vertices which implies that $\ZZ(\sigma)= \pm \ZZ$.

\vspace{2pt}
As with any filtered space, there is a spectral sequence with $E_1^{p,q} = \bar H_{p+q}(F_p - F_{p-1} ; \ZZ)$ converging to $\bar H_*(|C| ; \ZZ)$. Now  by Proposition \ref{cohocomp} we know what $E_1^{p,q}$ is for $p<N$.

\begin{prop}\label{ig}
$\bar H_*(|C| - F_N; \ZZ) \cong \bar H_*(|C|; \ZZ)$ for $* \ge 2e_d-N. $
\end{prop}
\begin{proof}
We first will try to bound  $\bar H_*(F_N ; \ZZ)$ and then use the long exact sequence of the pair $(F_N,|C|)$. The space $F_N$ is built out of locally closed subspaces $$\phi_k = \{(f, x_1,\dots,x_k), p|\: f \in \Sigma^{\ge N}, p \in\Delta^k, x_i \textrm{ are singular zeroes of } f   \}.$$  There exists a surjection $\pi: \phi_k \to \uconf_k X$. The map $\pi$ is a fibre bundle with fibres $\CC^{e_d -2k} \times \Delta^k_{\circ}$. The space $\uconf_k X$ has complex dimension $k$. Therefore  $\bar H_j(\phi_k; \ZZ) =0$ if  $ j  \ge 2e_d - N \ge 2(e_d -2k) +3k $. This implies $\bar H_j(F_N) = 0$ if $j\ge 2e_d -N$. The long exact sequence of the pair $(F_N , |C|)$ implies that $\bar H_*(Y - F_N; \ZZ) \cong \bar H_*(|C|; \ZZ)$ for $* \ge 2e_d-N $.
\end{proof}

Now this simplicial resolution of $\Sigma$ gives an associated spectral sequence for its Borel Moore homology with $$E_1^{p,q} = \bar H_{p+q} (F_p - F_{p-1}) = H_{p -( e_d - (2)(q+1) )} (\uconf_{p+1}(X), \pm \ZZ)$$, for $p< N$. Also, $E_1^{N,q} = 0$ if $q \ge 2e_d -N$.

We will now construct a cubical space $\CCa$ which will be involved in understanding $\Sigma(W)$. Our construction of $\CCa$ will be similar to that of $C$. Let $N = \frac{d-g}{2}$. 
Let $I$ be a subset of $\{1, \dots, N-1\}.$ Say $I = \{i_1, \dots, i_k\}$ let  $$\CCa_I :=\{(f, x_1, \dots, x_k)| f \in \Sigma(W), x_j \in \uconf_{i_j}(X), x_1 \subseteq \dots x_k  \subseteq \textrm{ Singular zeroes of  }f \}.$$
We define $$\CCa_{I \cup \{N\}}:= \{(f, x_1, \dots x_k) \in \CCa_I |  , f\in \bar \Sigma^{\ge N}(W)\}. $$ If $I \subseteq J$ then we have a natural forgetful map from $\CCa_J \to \CCa_I$. This gives $\CCa_{\cdot}$ the structure of a cubical space over the set $\{1, \dots, N\}$. We can take the geometric realization of $\CCa_{\cdot}$ denoted by $|\CCa|$. Then there is  a map $\rho: |\CCa| \to \Sigma(W)$, induced by the forgetful maps $\CCa_I \to \Sigma(W)$.

We again topologise $|\CCa|$ in a nonstandard way, this is entirely analogous to  the way we topologise $|C|$, so we will be brief in our description of it. We construct a bigger cubical  space $\bar \CCa$ such that for $I = \{i_1, \dots, i_k\} \subseteq \{1, \dots, N-1\}$, $\bar \CCa_I = \{ (f, x_1 \dots, x_k)| f \in \Sigma(\LL), \LL \in W, x_j \in L_{i_j}(\LL) , x_1 <x_2 \dots x_k< \textrm{Sing}(f)\}$. We define $\CCa_{I} \cup \{N\}$ analogously. We then form the geometric realisation, $|\bar \CCa|$ and note that there is a surjective map $|\bar \CCa| \to |\CCa|$ and we give $|\CCa|$ the quotient topology with respect to this map.

\begin{prop}
 $|\CCa|$ is proper homotopy equivalent to $\Sigma_W$.
\end{prop}
\begin{proof}
This is analogous to the proof of Proposition \ref{isequiv}.
\end{proof}

$\YY$ also has an ascending filtration, $\FF_n$ and this filtration gives us a spectral sequence for $\bar H_*(\YY ;\ZZ)$.

\begin{prop}
 $\bar H_*(|\CCa| - \FF_N; \ZZ) \cong \bar H_*(\CCa; \ZZ)$ for $* \ge 2e_d-N.$
\end{prop}
\begin{proof}
 This is analogous to the proof of Proposition \ref{ig}. 
\end{proof}

So there is a spectral sequence  with $$\EE_1^{p,q} = \bar H_{p+q} (\CCa_{0,\dots q}) = H_{p -( e_d - (2)(q+1) ) + 2g } (\uconf_{p+1}(X), \pm \ZZ)$$ for $p< N$. Finally we have the main theorem of this section.
\begin{thm}\label{key}
The map $\A^{-1}(\LL) \to \A^{-1}(W)$ induces an isomorphism $H_*(\A^{-1}(\LL); \ZZ) \to H_*(\A^{-1}(W); \ZZ)$ for $* < N = \frac{d- g}{2}$
\end{thm}
\begin{proof}
This proof  involves studying Alexander duality of $A^{-1}(\LL)$ inside $H^0(X, \LL)$ and $A^{-1}(W)$ inside $ H^0(X, W)$ (the space $H^0(X,W)$ is a topological vector bundle over $W$ and so is  at least homeomorphic to an affine space).

Now we use the fact that under Alexander duality, intersection of Borel-Moore cycles turns into pullback in cohomology, namely the map $$H^*(A^{-1}(W)) \to H^*(A^{-1}(\LL)),$$ is Alexander dual to the map $$f:\bar H_{*+2g}(\Sigma_W) \to \bar H_*(\Sigma_{\LL})$$ given by intersecting  cycles with $\Sigma_{\LL}$, i.e. $f(\sigma) = \sigma \cap \Sigma_{\LL}$.

To understand this map in Borel-Moore  homology, we turn to our spectral sequences for $\bar H_*(\Sigma_{\LL} ; \ZZ)$ and $\bar  H_*(\Sigma_W ; \ZZ)$. Since our stratification of $\Sigma_{W}$ is fiberwise we get a map of spectral sequences between the two spectral sequences. We have a map $\EE^1_{p, q+2g} \to E^1_{p,q}$. It will suffice to show that this map is an isomorphism for $p<N$. For $p<N$, this map is given by the map $$\phi :\bar H_{p+q +2g}(\FF_p - \FF_{p-1}, \ZZ) \to \bar H_{p+q}(F_p - F_{p-1}, \ZZ)$$ induced by intersecting cycles. However, we have a diagram of fiber bundles as follows:

\centerline{\xymatrix{
K \ar[rd] \ar[r] & K \times W \ar[rd] & \\
& F_p - F_{p-1} \ar[r] \ar[d] & \FF_p - \FF_{p-1} \ar[d]\\
& \uconf_p(X ) \ar@{=}[r] & \uconf_p{X }\\
}}

Using this diagram and the fact that the intersection map $\bar H_{*+2g}(K \times W) \to \bar H_{*}(K)$ is an isomorphism (as $W$ is homeomorphic to $\CC^g$), $\phi$ is an isomorphism. This implies the theorem.
\end{proof}

\section{Homology fibration theorem}

Let us recall the usual homology fibration theorem:
\begin{thm}[\cite{McDS}]\label{HFT}
Let $f: X \to Y$ be a map. Let $Hf^{-1}(y)$ be the homotopy fibre of $f$. Suppose $f^{-1}(y) \hookrightarrow f^{-1}(U)$ is a homology equivalence for sufficiently small $U$ open then $f^{-1}(y) \to Hf^{-1}(y)$ is a homology equivalence. 
\end{thm}

For this paper we need the following analogous theorem:

\begin{thm}\label{nHFT}
Let $n \ge 0$. Let $f : X \to Y$ be a map such that for all $y \in Y$ there exists an open neighbourhood $U$ such that the inclusion $j: f^{-1}(y) \hookrightarrow f^{-1}(U)$ induces an isomorphism $j_*: H_k(f^{-1}(y); \ZZ) \to H_k(f^{-1}(U); \ZZ) $ for all $k \le n$. 

Then the natural map $i: f^{-1}(y) \to Hf^{-1}(y)$ induces an isomorphism $i_*: H_k(f^{-1}(y); \ZZ) \to H_k(Hf^{-1}(y); \ZZ)$ for $k \le n$.
\end{thm}
\begin{proof}
This follows from the proof of Proposition 5 (which is the same as Theorem \ref{HFT} of this paper) in \cite{McDS}.   \end{proof}

This implies the following theorem.

\begin{thm}\label{eqh}
For $* \le N$, $$H^*(A^{-1}(\LL); \ZZ) \cong H^*(HA^{-1}(\LL);\ZZ) \cong H^*(\pi; \ZZ)$$ where $\pi$ is the previously described subgroup of the extended surface braid group. 
\end{thm}
\begin{proof}
First we note that $HA^{-1}(\LL)$ is a $K(\pi,1)$, where $\pi$ is our previously described subgroup of the extended surface braid group. The map $\A:U_n^{alg} \to Pic_n X $ has $K(G,1)$s for both source and target and the induced map at the level of $\pi_1$ is given by $\tilde Br_n X \to Br_n X \to \ZZ^{2g}$, Since $H \A^{-1} {\LL}$ is the homotopy fibre it is also a $K(G,1)$ with fundamental group equal to $$K_n : = \mathrm{ker}( \tilde{Br}_n X \to \ZZ^{2g}). $$

By Theorem \ref{nHFT} and Theorem \ref{key}  $H^*(A^{-1}(\LL); \ZZ) \cong H^*(H \A^{-1}(\LL) ; \ZZ)$  for $* \le N$.\end{proof}

\section{Relating $U(\LL)$ and $\UU(\LL)$}
In this section we will   relate the two spaces $U(\LL)$ and $\UU(\LL)$. We begin first with the following result.
\begin{prop}
Let $X$ be an algebraic curve. Let $\LL$ be a line bundle on $X$ of degree $n$. Let $p \in X$ be a point. Then,
$\UU(\LL) / \G_p \simeq U_{ n}^{alg}$.
\end{prop}
\begin{proof}
Let $$S_{\{x_1,\dots, x_n\}} = \{f \in \UU(\LL) : \textrm{ } x_i \textrm{ are regular zeroes of }f \} / \G_p.$$
We have a diagram of fiber bundles as follows.

$$\centerline{\xymatrix{
\CC^* \ar[rd] \ar[r]^f &  S_{\{x_1,\dots, x_n\}} \ar[rd] & \\
& U_{n}^{alg} \ar[d] \ar[r] & \UU(\LL)/ \G_p \ar[d]\\
&  \uconf_n X\ar@{=}[r] & \uconf_n X\\
}}$$

By considering the long exact sequences of homotopy groups associated to these fiber bundles,  it suffices to prove that the map $f: \CC^* \to S$ is a homotopy equivalence.

To prove this it suffices to prove that $S / \CC^*$ is contractible where $\CC^*$ is acting on $S$ by $z\cdot f(x) = z(f(x))$. But  $$ S / \CC^* = \{f \in \UU(\LL) | \textrm{ }x_i \textrm{ are regular zeroes of }f \} / \G.$$   This is contractible by (2)  of Proposition \ref{contG}.
\end{proof}

We will need the following lemma to obtain our results.

\begin{lem}\label{benson}
Let $X$ be an algebraic curve.  Let $n \ge 1$. Let $\bold a = \{a_1, \dots a_n\} \in \uconf_n X$. Let $\alpha \in \pi_1(\uconf_n X, \bold{a}) $. Suppose $\A_*(\alpha) \neq 0 \in  H_1(X ;\ZZ)$. Let $P_{\alpha}$ denote the point-pushing map associated to $\alpha$. Let $c_i \in H_1(X - \{a_1 \dots a_n\} ; \ZZ)$ be the puncture classes. Then there exists a class $$\gamma \in H_1(X -\{a_1 \dots a_n\} ; \ZZ)$$ such that $\gamma \cap \A_*(\alpha) =1$ and  $$P_{\alpha}(\gamma) - \gamma = \sum m_i c_i,$$ where $\sum m_i \neq 0$.
\end{lem}
This can be deduced from a computation by Bena Tshishiku. For a refererence see \cite{FKW}.

\begin{thm}\label{nullh}
The natural map $\rho: \UU(\LL) \to \textrm{Pic}_n X$ is nullhomotopic.
\end{thm}
\begin{proof}
As $\textrm{Pic}_n X$ is a $K(\pi,1)$ it suffices to prove that $$\rho_*: \pi_1(\UU(\LL)) \to  \pi_1(\textrm{Pic}_n X) \cong H_1(X; \ZZ)$$ is trivial. Let $\pi: \UU(\LL) \to \uconf_n X$ be defined by $ \pi(s) = \{a \in X | s(a) = 0\}.$ Let $\A: \uconf_n X \to \textrm{Pic}_n X$ be the Abel-Jacobi map. Let $\bold a = \{a_1, \dots a_n\} \in \uconf_n X$. Let $\alpha \in \pi_1(\uconf_n X, \bold{a}) $. Suppose $\A^* (\alpha) \neq 0 \in H_1(X ;\ZZ)$. It suffices to show that $\alpha \not \in \rho_*(\pi_1(\UU(\LL))).$ This is because the map $\rho$ factors through $\A$.

Let $Mod(X - \{a_1, \dots, a_n\})$ be the mapping class group of the punctured surface $X - \{a_1, \dots, a_n\}$. Associated to $\alpha$ there exists a point pushing map $P_\alpha \in Mod(X - \{a_1, \dots a_n\})$. 
Let $c_i \in H_1(X - \{a_1 \dots a_n\} ; \ZZ)$ be the puncture classes.

 Then by Lemma \ref{benson}there exists a class $[\gamma] \in H_1(X - \{a_1 \dots a_n\})$ such that $$P_{\alpha_*}([\gamma]) -[\gamma] = \sum m_ic_i,$$ where $m_i \in \ZZ$ satisfying $\sum m_1 \neq 0$.

 Let $f \in \UU(\LL)$ be such that $\pi (f) = \bold{a}$.
 Suppose for the sake of contradiction that $\alpha \in \mathrm{im}(\pi_1(\UU(\LL)),f)$ with $\alpha \neq 0$. Then there exists a loop in $\UU(\LL)$, which we will call $F_{\alpha}$  such that $\pi(F_\alpha) = \alpha$, i.e. $F_{\alpha}$ is a lift of $\alpha$.
 
 Now for $s \in (0,1)$, let $P^s_{\alpha}$ be the point-pushing homeomorphism  along the path $\alpha|_{[0,s]}$. It is a well-defined element of $$\pi_0(\textrm{Homeo}((X, \alpha(0)),(X, \alpha(s)))).$$ 
 Now $P^t_{\alpha}(f)$ is a lift of $\alpha$ as a path (not a loop) to $\UU(\LL)$. Since the map $\pi$ is a fibration, any two paths that are lifts of $\alpha$ must have endpoints in the same component of $\pi^{-1}(\bar{a})$. This would imply that $f$ (the endpoint of $F_{\alpha}$) and $P_{\alpha_*}(f)$ (the endpoint of $P^t_{\alpha}(f)$) would be in the same path component of $\pi^{-1}(\bold{a})$. If that were so, then we would have 
 $$\int_{\gamma}\frac{ P_{\alpha}(f)}{|P_{\alpha}(f)|}- \int_\gamma \frac{f}{|f|}=0.$$
 However we will now show that this is not the case.
 
  We'd like  to remind the reader that $\int_{c_i} f/ |f| =1$. This is because the section $f$ has a zero of index $1$ at each of the $a_i$s. Then we know that
$$\int_{\gamma}\frac{ P_{\alpha}(f)}{|P_{\alpha}(f)|}- \int_\gamma \frac{f}{|f|}= \int_{P_{\alpha_*}\gamma}  \frac{f}{|f|} -\int_\gamma \frac{f}{|f|}$$

$$
 =\sum_i m_i \int_{c_i}f/ |f| = \sum m_i \neq 0.$$
 This proves that any $\alpha$ that lifts is forced to be trivial, which completes the proof.
\end{proof}

\begin{prop}\label{classmap}
let $n \ge 1$, $p \in X$. There exists a homotopy equivalence $f : \textrm{Pic}_n X \to B \G_p$ that makes the following diagram commute upto homotopy:

\centerline{ \xymatrix{
     & \UU(\LL) \ar[d]^\pi\\
U_n^{alg} \ar[d]^{\A} \ar[r]^i & \UU(\LL)/ \G_p \ar[d]^ g \\
\textrm{Pic}_n X \ar[r]^f & B\G_p\\
}}
Here the map $g$ is the classifying map for the fibration $\UU(\LL) \to \UU(\LL) / \G_p$.
\end{prop}
 \begin{proof}
 The situation is a follows: 
 $\pi: \UU(\LL) \to \UU(\LL)/ \G_p$ is a principal $\G_p$ bundle. Since $\G_p \simeq \ZZ^2g$, we have an associated principal $\ZZ^{2g}$ bundle $E \to \UU(\LL)/ \G_p$, where $$E:= \UU(\LL) / (f_1 \sim f_2 \textrm{ if } \pi(f_1) = \pi(f_2) \textrm{  and } f_1, f_2 \textrm{are in the same path component of } \pi^{-1}(\pi(f_1))).$$
 Equivalently, if $(\G_p)_0$ is the identity component of $\G_p$, $E = \UU(\LL)/(\G_p)_0$.
 
 The quotient map $p:\UU(\LL) \to E$ is naturally a homotopy equivalence as the group $(\G_p)_0$ is contractible.
 We then have a diagram as follows:

 \centerline{
 \xymatrix{
\UU(\LL) \ar[d] \ar[r]^p & E \ar[d]\\
\UU(\LL)/ \G_p \ar[r]^= & \UU(\LL)/ \G_p \\
}}
 
 It suffices to prove that the natural map $$\alpha: \UU(\LL)/ \G_p \to \textrm{Pic}_n X$$ satisfies the classifying space property for the fibration $E \to \UU(\LL)/ \G_p$.  However by Proposition \ref{nullh} the composite map $E \to Pic_n X$ is nullhomotopic and we can lift it to $\tilde{\textrm{Pic}_n X}$, the universal cover of $\textrm{Pic}_n X$.
 So we have a commutative diagram as follows.
 
 \centerline{
 \xymatrix{
E \ar[d] \ar[r] & \tilde{\textrm{Pic}_n X} \ar[d]\\
\UU(\LL) / \G_p\ar[r]^\alpha & \textrm{Pic}_n X\\
}}
Hence $\alpha$ is a classifying map and we are done. 
\end{proof}

\begin{thm}\label{UUish}
$\UU(\LL)$ is homotopy equivalent to $H\A^{-1}(\LL)$, the homotopy fibre of $\A$. 
\end{thm}
\begin{proof}
By Propositions \ref{nullh} and \ref{classmap} there is a diagram as follows:

\centerline{
\xymatrix{
H\A^{-1}(\LL) \ar[d] & \UU(\LL) \ar[d]\\
U_n^{alg} \ar[d]^{\A} \ar[r]^i & \UU(\LL)/ \G_p \ar[d]\\
\textrm{Pic}_n X \ar[r]^f & B\G_p\\
}}
Since the composite map $H \A^{-1}(\LL) \to B\G_p$ is null homotopic, by the properties of a fibre sequence we have a map $g:H\A^{_1}(\LL) \to \UU(\LL)$ that commutes with the maps of the diagram. Since the maps $i$ and $f$ are homotopy equivalences, so is $g$.
\end{proof}

Now we can finally prove the theorems in the introduction of this paper.
\begin{proof}[Proof of Theorem \ref{main2}]
By Theorem \ref{UUish} $\UU(\LL) \simeq H \A^{-1}(\LL).$ So it suffices to prove that $H \A^{-1}(\LL)$ is a $K(\pi,1)$ for $K_n$. However this follows from Theorem \ref{eqh}.
\end{proof}

\begin{proof}[Proof of Theorem \ref{main}]
By Theorem \ref{nHFT} and Theorem \ref{key} the map $f :U(\LL) \to H \A^{-1}(\LL)$ induces an isomorphism $H^*(U (\LL)) \cong H^*(H \A^{-1}(\LL))$ for $*<n -(2g))$. But by  Theorem \ref{UUish} $H \A^{-1} (\LL) \simeq  \UU(\LL)$ and it is easy to see that $$i^*: H^*(\UU(\LL); \ZZ) \to H^*(U(\LL); \ZZ)$$ is the composition $$
H^*(\UU(\LL); \ZZ) \cong H^* ( H\A^{-1}(\LL);\ZZ) \rightarrow^{f^*} H^*(U(\LL); \ZZ).   
$$
\end{proof}

\end{document}